\documentclass[11pt]{amsart}
\usepackage{amsmath}
\usepackage{amsfonts}
\usepackage{amsthm}
\usepackage{amssymb}
\usepackage{mathrsfs}
\usepackage{latexsym}
\usepackage{verbatim}
\usepackage{diagrams}
\usepackage{a4wide}
\usepackage{enumerate}
\usepackage{hyperref}

\renewcommand{\thetheoremName}

\def\aug{\mathfrak{a}}

\def\Ad{\mathrm{ad}}

\def\RS{R_{\mathcal{S}}}

\def\Rloc{R^{\mathrm{loc}}}

\def\alphajj{\alpha_{j\kern-0.04em{j}}}

\def\CO{\mathcal{O}}

\newarrow{Equals}{=}{=}{=}{=}{=}
\newarrow{pots}{.}{.}{.}{.}{>}
\def\m{\mathfrak{m}}

\def\Z{\mathbf{Z}}

\def\OL{\mathcal{O}}
\def\Q{\mathbf{Q}}
\def\F{\mathbf{F}}

\def\PGL{\mathrm{PGL}}

\def\Hom{\mathrm{Hom}}

\def\GL{\mathrm{GL}}
\def\rank{\mathrm{rank}}

\def\Qbar{\overline{\Q}}

\def\rhobar{\overline{\rho}}

\def\SL{\mathrm{SL}}

\def\eps{\epsilon}

\def\Frob{\mathrm{Frob}}
\def\Spec{\mathrm{Spec}}

\def\tij{t_{i\kern-0.03em{j}}}

\def\CO{\mathcal{O}}

\newtheorem{theorem}{Theorem} 

\newtheorem{lemma}[theorem]{Lemma}
\newtheorem{prop}[theorem]{Proposition}
\newtheorem{corr}[theorem]{Corollary}

\newtheorem{remark}[theorem]{Remark}

\def\Ru{R^{\mathrm{unr}}}
\def\Rsplit{R^{\mathrm{split}}}

\def\Runiv{R^{\mathrm{univ}}}

\def\Run{R^{\mathrm{un}}}

\def\Gal{\mathrm{Gal}}
\def\rbar{\overline{r}}

\newcommand{\UZ}{U \kern-.1em{Z}}

\newcommand{\VF}{V \kern-.07em{F}}

\def\rhounr{\rho^{\mathrm{unr}}}

\begin{document}

\title{Finiteness of unramified deformation rings}
\author{Patrick B. Allen and Frank Calegari}
\thanks{The first author was supported in part by a Simons Research Travel Grant.
The second author was supported in part by NSF Career Grant
  DMS-0846285.}

\begin{abstract}
We prove that the universal unramified deformation ring $\Ru$ of a continuous Galois representation $\rhobar: G_{F^{+}} \rightarrow \GL_n(k)$ (for a totally real field~$F^{+}$ and finite field~$k$) 
is finite over $\OL = W(k)$ in many cases. We also
prove (under similar hypotheses) that the universal deformation ring $\Runiv$ is finite over the local deformation ring $\Rloc$.
\end{abstract}

\maketitle

\section*{Introduction}

Let $k$ be a finite field of characteristic $p$, and let $\OL = W(k)$. 
Let $F$ be a number field, 
and consider a  continuous absolutely irreducible Galois representation
$$\rhobar: G_{F} \rightarrow \GL_n(k),$$
where $G_F = \Gal(\overline{F}/F)$ for some fixed algebraic closure $\overline{F}$ of $F$.
If $(A,\m)$ is a complete local  $\OL$-algebra with residue field $k$, then a
deformation
$\rho$ of $\rhobar$  to $A$ unramified outside a finite set of primes $S$ consists an equivalence class of homomorphisms
$$ \rho: G_{F} \rightarrow \GL_n(A)$$
such that the composite of $\rho$ with the projection
$\GL_n(A) \rightarrow \GL_n(A/\m) = \GL_n(k)$ is $\rhobar$, and such that the extension of fields
$F(\ker(\rho))$ over $F(\ker(\rhobar))$ is unramified away from places above primes in $S$ (see~\cite{MazurDeform}). 
The nature of such deformations is quite different depending on whether $S$ contains the primes above $p$ or not.
If $S$ contains all the primes above $p$, we denote the  universal deformation ring by $\Runiv$; if $S$
contains no primes above $p$, we denote the corresponding universal deformation ring by $\Ru$.
According to the  Fontaine--Mazur conjecture (see~\cite{FM}, Conj.~5a), any
map $\Ru \rightarrow \Qbar_p$ gives rise to a deformation  $\rho$ of $\rhobar$ with finite image. (This form
of the conjecture is known as the unramified Fontaine--Mazur conjecture.) 
Boston's strengthening of this conjecture (\cite{BostonII}, Conjecture~2 and the subsequent corollary) is the claim that
the \emph{universal} unramified deformation:
$$\rhounr: G_{F} \rightarrow \GL_n(\Ru)$$
has finite image.
In contrast, the ring $\Runiv$  is typically of large dimension (see \S~1.10 of \cite{MazurDeforming}). 
A conjecture of Mazur 
predicts that the relative dimension of $\Runiv$ over $\OL$ is
(in odd characteristic)
$$(1 + r_2) + (n^2 - 1)[F:\Q] - \sum_{v|\infty} 
\dim H^0(D_{v},\Ad^0(\rhobar)),$$
where
$\Ad^0(\rhobar)$ denotes (in any choice of basis) the trace zero matrices in $\Hom(\rhobar,\rhobar)$. 
A choice of basis for the universal deformation makes $\Runiv$
an algebra over a local deformation ring 
$$\Rloc = \displaystyle{\widehat{\bigotimes}_{v|p} \Rloc_v}$$
where $\Rloc_v$ is the universal framed local deformation ring of $\rhobar|{D_v}$ for $v|p$. The $\Rloc$-algebra structure may depend on the choice of basis, but it is canonical up to automorphisms of $\Rloc$.
  It is not true in general that
$\Spec(\Runiv) \rightarrow \Spec(\Rloc)$ is a closed immersion, even in the minimal case where $S$ is  only divisible by the  primes dividing~$p$. A simple example to consider
is the deformation ring of any one dimensional representation $\rhobar: G_{F} \rightarrow k^{\times}$;
the corresponding map $\Spec(\Runiv) \rightarrow \Spec(\Rloc)$ is a closed immersion
if and only if the maximal everywhere unramified abelian~$p$-extension of~$F$ in which~$p$
splits completely is trivial.
  It is, however,  reasonable to  conjecture that this map is always 
a finite morphism.  Indeed, one  heuristic justification for the  Fontaine--Mazur conjecture is to imagine that the
generic fibres of the image of $\Spec(\Runiv)$ and the locus of local crystalline representations of a fixed weight are transverse, and to
infer (from a conjectural computation of dimensions) that the intersection is finite, and hence that there are only
finitely many global crystalline representations of a fixed weight (see pg.~191--192 of \cite{FM}); this line of thinking at least presumes that the global to local map is quasi--finite.

We prove the following:

\begin{theorem} \label{theorem:one} Let $F^{+}$ be a totally real field, and let $\rhobar: G_{F^{+}} \rightarrow \GL_n(k)$ be a continuous absolutely irreducible
representation. Suppose that:
\begin{enumerate}
\item $p >  2$.
\item $\Ad^0(\rhobar|G_{F^{+}(\zeta_p)})$ is absolutely irreducible and $p > 2n^2 - 1$, or, if $n = 2$ and $\rhobar$ is totally odd,
$\rhobar|G_{F^{+}(\zeta_p)}$ has adequate image. 
\end{enumerate}
Then  $\Ru$ is a finite $\OL$-algebra, and  $\Runiv$ is a finite $\Rloc$-algebra.
\end{theorem}

The second condition holds, for example, when $\rhobar$ has image containing $\SL_n(k)$ and~$p$ is greater than
  $2n^2 - 1$. 
The finiteness of $\Runiv$ over $\Rloc$ can be deduced from appropriate ``$R = \mathbf{T}$" theorem, since one proves that the maximal reduced quotient of $\Runiv$ modulo an ideal of $\Rloc$ is isomorphic to a finite $\CO$-algebra $\mathbf{T}$.
However in dimension $>2$ without a conjugate self dual assumption, the current ``$R = \mathbf{T}$" theorems are contingent on conjectural properties of the cohomology of arithmetic quotients (see Part~2 of \cite{CG}).

We shall  deduce from Theorem~\ref{theorem:one}  the following corollaries:

 \begin{corr} \label{corr:boston} For any $\rhobar$ satisfying the conditions of Theorem~\ref{theorem:one}, Boston's
 strengthening of the unramified Fontaine--Mazur conjecture is equivalent to the unramified Fontaine--Mazur conjecture.
 \end{corr}
 
 \begin{corr} \label{corr:FM} Suppose that $\rhobar: G_{F^{+}} \rightarrow \GL_2(k)$ satisfies the conditions of Theorem~\ref{theorem:one}. Assume futher that:
 \begin{enumerate}
 \item $\rhobar$ is totally odd.
 \item If $p = 5$ and $\rhobar$ has projective image
 $\PGL_2(\F_5)$, then $[F^{+}(\zeta_5):F^{+}] = 4$. 
 \end{enumerate}
 Then Boston's conjecture holds; the
representation $\rhounr: G_{F^{+}} \rightarrow \GL_2(\Ru)$ has finite image.
\end{corr}

When $n = 2$, $p > 2$, $F = \Q$, and $\rhobar$ is totally odd and unramified at $p$,
 $\Ru$ can be identified  with the ring of Hecke operators acting on a (not necessarily torsion free)
coherent cohomology group (see~\cite{CG}).
 
\medskip


Let $\mathcal{G}_n$ be the group scheme over $\Z$ that is the semidirect product
	\[ (\GL_n \times \GL_1) \rtimes \{1,\jmath\} = \mathcal{G}_n^0 \rtimes \{1,\jmath\} \]
where $\jmath$ acts on $\GL_n\times \GL_1$ by $\jmath (g,\mu)\jmath^{-1} = (\mu{}^tg^{-1},\mu)$. Let $\nu: \mathcal{G}_n \rightarrow \GL_1$ be the character that sends $(g,\mu)$ to $\mu$ and $\jmath$ to $-1$. Let $F$ be a CM field with maximal totally real subfield $F^+$, and let 
	\[\rbar : G_{F^+} \rightarrow \mathcal{G}_n(k)\]
be a continuous homomorphism with $\rbar^{-1}(\mathcal{G}_n^0(k)) = G_{F}$. If $(A,\m)$ is a complete local  
$\OL$-algebra with residue field $k$, then a deformation
$r$ of $\rbar$ to $A$ unramified outside a finite set of primes $S$ consists an equivalence class of homomorphisms
\[ r: G_{F^+} \rightarrow \mathcal{G}_n(A)\]
such that the composite of $r$ with the projection
$\mathcal{G}_n(A) \rightarrow \mathcal{G}_n(A/\m) = \mathcal{G}_n(k)$ is $\rbar$, and such that the extension of fields
$F(\ker(r))$ over $F(\ker(\rbar))$ is unramified away from places above primes in $S$. We say two lifts are equivalent if they are conjugate by an element of $\GL_n(A)$ that reduces to the identity modulo $\m$. If $\rbar$ is Schur (see Definition 2.1.6 of \cite{CHT}), then this deformation problem is representable. By abuse of notation, we will again denote the  universal deformation ring of $\rbar$ by $\Runiv$ if $S$ contains all the primes above $p$, and by $\Ru$ if $S$
contains no primes above $p$. This shouldn't cause any confusion, as we shall be very explicit  as to which deformation problem we are refering.
As with the $\GL_n$-valued theory, for each $v|p$ in $F^+$, there is a universal framed deformation ring $R_v^\square$ which represents the lifts of $\rbar|D_v$, and a choice of lift in the equivalence class of the universal deformation of $\rbar$ makes $\Runiv$ an algebra over 
	\[ \Rloc = \displaystyle{\widehat{\bigotimes}_{v|p} \Rloc_v}.\]
We shall deduce Theorem~\ref{theorem:one} from the following result.

\begin{theorem}  \label{theorem:main} Let $F$ be a CM field with maximal totally real subfield $F^+$. Let $S$ denote a finite set of places of $F^+$ not containing any $v|p$, 
and let $\overline{r}: G_{F^+} \rightarrow \mathcal{G}_n(k)$ be a continuous homomorphism with $\overline{r}^{-1}(\mathcal{G}_n^0(k)) = G_{F}$ and such that $\nu\circ\rbar(c_v) = -1$ for each choice of complex conjugation $c_v$. Assume that $p\ge 2(n+1)$, that the image of
$\rbar|G_{F(\zeta_p)}$ is adequate, and that 
$\zeta_p \notin F$. Let $\Ru$ be the universal deformation ring of $\rbar$ unramified outside $S$,
and let $\Runiv$ be the universal deformation ring of $\rbar$ unramified outside $S$ and all
primes $v|p$. Then $\Ru$ is a finite $\OL$-algebra, and
 $\Runiv$ is a finite $\Rloc$-algebra.
\end{theorem}

It turns out that the proof of this theorem is almost an immediate consequence of the finiteness results of~\cite{thorne}
for ordinary deformation rings. The only required subtlety is to understand the relationship
between the local ordinary deformation ring $R_{\Lambda_K}^{\triangle,ar}$  constructed by~\cite{G}  and the unramified local deformation
ring $\Run$.


\medskip

We would like to thank Matthew Emerton, Toby Gee, and Vytautas Paskunas for useful conversations.  We would like to 
thank the organizers of the~2014 Bellairs workshop in Number Theory, where some of the ideas
in this note were first discussed. 
We would also like to thank the referees for helpful comments on a previous version of this note.

\section{Some Local Deformation Rings}\label{sec:localdefring}

Recall $k$ is a finite field of characteristic $p$, and $\OL = W(k)$. Let $K$ be a finite extension of $\Q_p$ and let $G_K = \mathrm{Gal}({\overline{K}/K)}$. Fix a continuous unramified representation
	\[\rhobar: G_K \rightarrow \GL_n(k)\]
and let $R^\square$ be its universal framed deformation ring. Let $\Run$ be the quotient of $R^\square$ corresponding to unramified lifts. 

\begin{lemma}\label{lemma:unram}
The ring $\Run$ is isomorphic to a power series ring over $\CO$ in $n^2$ variables. In particular, it is reduced and its $\Qbar_p$-points are Zariski dense in $\Spec(\Run)$.
\end{lemma}

\begin{proof} Fixing a choice of lift $g\in \GL_n(\CO)$ of $\rhobar(\mathrm{Frob})$, it is easy to see that the lift to $\CO[[\{x_{ij}\}_{1\le i,j\le n}]]$ given by $\Frob \mapsto g(I+(x_{ij}))$ is the universal framed deformation.
\end{proof}

Let $I_K^{\mathrm{ab}}$ be the inertia subgroup of the abelianization of $G_K$, and let $I_K^{\mathrm{ab}}(p)$  be its maximal pro-$p$ quotient. Let $\Lambda_K = \OL[[(I_K^{\mathrm{ab}}(p))^n]]$ and let $\psi = (\psi_1,\ldots,\psi_n)$ be the universal $n$-tuple of characters $\psi_i : I_K \rightarrow \Lambda_K^\times$. Set $R_{\Lambda_K}^\square = R^\square \widehat{\otimes}_{\OL}\Lambda_K$. 

We briefly recall the construction of the universal ordinary deformation ring $R_{\Lambda_K}^\triangle$ by Geraghty (see \S~3 of \cite{G}). Let $\mathcal{F}$ be the flag variety over $\OL$ whose $S$ points, for any $\OL$-scheme $S$, is the set of increasing filtrations $0 = F_0\subset F_1\subset \cdots \subset F_n = \OL_S^n$ of $\OL_S^n$ by locally free submodules with $\rank(F_i) = i$ for each $i=1,\ldots,n$. Lemma 3.1.2 of \cite{G} shows that the subfunctor of
	\[ R_{\Lambda_K}^\square\otimes_{\OL} \mathcal{F}\]
corresponding to pairs $(\rho, \{F_i\})$ such that 
	\begin{itemize}
		\item $\{F_i\}$ is stabilized by $\rho$ and
		\item the action of $I_K$ on $F_i/F_{i-1}$ is given by the pushforward of $\psi_i$,
	\end{itemize}
is represented by a closed subscheme $\mathcal{G}$. He then defines $R_{\Lambda_K}^\triangle$ as the image of
	\[ R_{\Lambda_K}^\square \rightarrow \OL_{\mathcal{G}}(\mathcal{G}[1/p]).\]
	
Since scheme theoretic image commutes with flat base change, $R_{\Lambda_K}^\triangle[1/p]$ is the scheme theoretic image of 
	\[ \mathcal{G}[1/p] \rightarrow \Spec (R_{\Lambda_K}^\square[1/p]).
	\]
Since this map is proper, $\mathcal{G}[1/p]$ surjects onto $\Spec (R_{\Lambda_K}^\triangle[1/p])$. Because $\mathcal{G}$ is finite type over $R_{\Lambda_K}^\triangle$, we deduce that any $\Qbar_p$ point of $\Spec (R_{\Lambda_K}^\triangle[1/p])$ lifts to a $\Qbar_p$-point of $\mathcal{G}[1/p]$. This proves the following.

\begin{lemma}\label{lemma:trianglepoints} Let $x\in \Spec (R_{\Lambda_K}^\square)(\Qbar_p)$, and let $(\rho_x,\psi_x)$ denote the pushforward via $x$ of the universal framed deformation and $n$-tuple of characters of $I_K$. Then $x$ factors through $R_{\Lambda_K}^\triangle[1/p]$ if and only if there is a full flag $0=F_0\subset F_1\subset \cdots \subset F_n=\Qbar_p^n$ stabilized by $\rho_x$ such that the action of $I_K$ on $F_i/F_{i-1}$ is given by $\psi_{i,x}$ for each $i=1,\ldots,n$.
\end{lemma}

If $\rhobar$ is the trivial representation, then Geraghty defines a further quotient $R_{\Lambda_K}^{\triangle,ar}$ of $R_{\Lambda_K}^\triangle$ as follows. Let $Q_1,\ldots,Q_m$ be the minimal primes of $\Lambda_K$. For each $j = 1,\ldots,m$, let $\mathcal{G}_j = \mathcal{G}\otimes_{\Lambda_K}\Lambda_K/Q_j$. Let $W_j\subset\Spec(\Lambda_K/Q_j)$ be the closed subscheme defined by $\psi_r = \epsilon_p\psi_s$ for some $1\le r < s \le n$, and let $U_j$ be the complement of $W_j$. Geraghty shows (see \S~3.4 of \cite{G}) that there is a unique irreducible component $\mathcal{G}_j^{ar}$ of $\mathcal{G}_j$ lying above $U_j$. We then set $\mathcal{G}^{ar} = \cup_{1\le j\le m}\mathcal{G}_j^{ar}$ and define  $R_{\Lambda_K}^{\triangle,ar}$ to be the image of
	\[ R_{\Lambda_K}^{\triangle}\rightarrow \OL_{\mathcal{G}^{ar}}(\mathcal{G}^{ar}[1/p]).\]
	
The construction of $R_{\Lambda_K}^{\triangle,ar}$ together with Lemma \ref{lemma:trianglepoints} yields the following.

\begin{lemma}\label{lemma:arpoints}
Assume that $\rhobar$ is trivial. Let $x\in \Spec (R_{\Lambda_K}^\square)(\Qbar_p)$, and let $(\rho_x,\psi_x)$ denote the pushforward via $x$ of the universal framed deformation and $n$-tuple of characters of $I_K$. Assume that there is a full flag $0=F_0\subset F_1\subset \cdots \subset F_n= \Qbar_p^n$ stabilized by $\rho_x$ such that the action of $I_K$ on $F_i/F_{i-1}$ is given by $\psi_{i,x}$ for each $i=1,\ldots,n$. Assume further that $\psi_{i,x} \ne \epsilon_p\psi_{j,x}$ for any $i<j$. Then $x$ factors through $R_{\Lambda_K}^{\triangle,ar}$.
\end{lemma}

\begin{remark}  \emph{If $[K:\Q_p] > n(n-1)/2 + 1$  and $\rhobar$ is trivial (which, for our applications, we could assume),
then Thorne proves  that 
$R_{\Lambda_K}^{\triangle,ar} = R_{\Lambda_K}^{\triangle}$ 
(see Corollary~3.12 of~\cite{thornered}).}
\end{remark}

There is a natural map $\Lambda_K \rightarrow \Run$ given by modding out by the augmentation ideal
$\aug$ of $\Lambda_K$. We thus have a natural surjection
	\[ R_{\Lambda_K}^\square \rightarrow \Run.\]

\begin{prop}\label{prop:unramord} The surjection $R_{\Lambda_K}^\square \rightarrow \Run$ factors through $R_{\Lambda_K}^\triangle$. If $\rhobar$ is trivial, then it further factors through $R_{\Lambda_K}^{\triangle,ar}$.
\end{prop}

\begin{proof}  
The image of an unramified representation is the topological closure of the
image of Frobenius. Since any element of $\GL_n(\Qbar_p)$ is conjugate to an upper triangular matrix,  that the image of any unramified representation into~$\GL_n(\Qbar_p)$ fixes a full flag for which
the action of inertia on the corresponding quotients is trivial.
It follows that the projection from $R_{\Lambda_K}^\square$ to any~$\Qbar_p$-point  of $\Run$ factors through 
$R_{\Lambda_K}^\triangle$ by Lemma~\ref{lemma:trianglepoints} and, if $\rhobar$ is trivial,
through $R_{\Lambda_K}^{\triangle,ar}$, by Lemma~\ref{lemma:arpoints}.
The result then follows from the fact
that~$\Run$ is reduced and its $\Qbar_p$-points are Zariski dense,
by  Lemma \ref{lemma:unram}.
 \end{proof}

\section{Proof of  Theorem~\ref{theorem:main}}

We first prove the statement concerning $\Ru$ over $\OL$.
Take a representation $\rbar$ as in Theorem~\ref{theorem:main}.
For eavh $v|p$ in $F^+$, let $F_v^+$ be the completion of $F^+$ at $v$ and let $\Lambda_v = \Lambda_{F_v^+}$ with $\Lambda_{F_v^+}$ as in \S~\ref{sec:localdefring}. Let $\Lambda = \widehat{\otimes}_{v|p,\OL}\Lambda_v$.

We note that using Lemma~1.2.2 of~\cite{BLGGT}, we are free to make any base change disjoint from the fixed field of $\ker(\rbar)$. After a base change, we may assume that $\rbar$ is everywhere unramified, and that $\rbar |D_{v}$ is trivial for all $v|p$ as well as any finite set of auxiliary primes.
In particular, after a suitable base change, we may restrict ourselves to considering
deformation rings which are unipotent at some finite set of auxiliary primes~$v \in S$
(which corresponds to the local deformation condition $R^{1}_{\widetilde{v}}$ of~\cite{thorne},~\S~8).
By Proposition~3.3.1 of~\cite{BLGGT},
 we may assume, after a further base change, that $\rbar$ lifts to a minimal crystalline
ordinary modular representation (this is where we use the assumption that $p\ge 2(n+1)$). 
 From Corollary~8.7 of~\cite{thorne}, we deduce that the corresponding ordinary  deformation ring $\RS$ is finite over $\Lambda$. If we can  show that $\Ru$ is a quotient of $\RS \otimes \Lambda/\aug$, where
 $\aug$ is the augmentation ideal of~$\Lambda$, then 
the result follows immediately by Nakayama, since $\Lambda/\aug = \OL$. 
By definition, the local condition at $v|p$ for $\RS$ is determined by  the ordinary deformation ring
$R_{\Lambda_v}^{\triangle,ar}$. By
Proposition~\ref{prop:unramord}, the ring $\Run$ is a quotient of $R_{\Lambda_v}^{\triangle,ar}/\aug$.
Hence $\Ru$ is a quotient of $\RS \otimes \Lambda/\aug$ and we are done.

The finiteness of $\Runiv$ over $\Rloc$ then follows from the finiteness of $\Run$ over $\OL$ and Nakayama. Indeed, let $\Rsplit = \Runiv \otimes_{\Rloc}k$ and let $r^{\mathrm{split}}$ be the specialization of the universal deformation to $\Rsplit$. Then $r^{\mathrm{split}}|D_v \cong \rbar|D_v$ for any $v|p$ in $F^+$, so the quotient $\Runiv \rightarrow \Rsplit$ factors through $\Run\otimes_{\OL}k$.

\section{Some Corollaries}

\subsection{Proof of Theorem~\ref{theorem:one}}

Let $\rhobar$ satisfy the statement of Theorem~\ref{theorem:one}.
Consider $\Ad^0(\rhobar)$ restricted to a suitable quadratic CM extension $F/F^{+}$.
Since $p \nmid n$, the representation $\Ad^0(\rhobar)$ is a direct summand of 
$\rhobar^c \otimes \rhobar^* = \rhobar \otimes \rhobar^{*}$
and is conjugate self-dual. The  assumption of irreducibility together
with the inequality $p > 2n^2 -1$ imply
 that~$\Ad^0(\rhobar)$
is adequate by Theorem~A.9 of~\cite{thorne}.
If $n$ is even, then $\Ad^0(\rhobar)$ has odd dimension and so is automatically totally odd. 
If $n$ is odd, then $\Ad^0(\rhobar)$ is orthogonal (the conjugate self-duality is realized by the trace pairing, which is symmetric) and exactly self-dual (up to trivial twist) and so has trivial multiplier, which
means that is also
totally odd. Both uses of totally odd refer to the properties of the multiplier character rather
than the determinant of complex conjugation, and are the exact  sign conditions required
for automorphy lifting theorems for unitary groups (that is, totally odd means $U$-odd rather than
$\GL$-odd in the notation of~\cite{sign}, see also~\S~2.1 of~\cite{BLGGT}).
Hence $\Ad^0(\rhobar)|G_F$ extends to a homomorphism (see Lemma 2.1.1 of \cite{CHT})
	\[ \rbar : G_{F^+} \rightarrow \mathcal{G}_{n^2-1}(k), \]
which we fix, satisfying the conditions of Theorem~\ref{theorem:main}. On the other hand, any deformation of $\rhobar$
gives rise to a deformation of $\rbar$ in the natural way. 
By Yoneda's Lemma, there is a corresponding morphism $\Ru({\rbar}) \rightarrow \Ru({\rhobar})$. It suffices to
prove this is finite. By Nakayama's Lemma, this reduces to showing that the only deformations $\rho$ of $\rhobar$ to $k$-algebras such that $\Ad^0(\rho)|G_F \cong \Ad^0(\rhobar)|G_F$ are finite. The kernel of such a deformation must be contained in the maximal abelian pro-$p$ extension of $F(\ker(\rhobar))$ unramified outside $S$, which is finite by class field theory. 
As in the final paragraph of the proof of Theorem~\ref{theorem:main}, the finiteness
of $\Ru$ implies the finiteness of $\Rsplit$ and hence that $\Runiv$ is a finite $\Rloc$-algebra.

If $n = 2$ and $\rhobar$ is totally odd, then we may work directly with $\rhobar$. We first use Corollary~1.7 of~\cite{TaylorRmk} to conclude that $\rhobar$ is potentially modular and Theorem~A of \cite{BLGGord2} to assume it is potentially ordinarily modular. Then, restricting $\rhobar$ to a suitable CM field $F$, the proof is exactly as in the proof of Theorem~\ref{theorem:main} (without the appeal to Proposition~3.3.1 of~\cite{BLGGT}).


\subsection{Proof of Corollary~\ref{corr:boston}}
This follows immediately from Theorem~\ref{theorem:one} and the following proposition.

\begin{prop}\label{prop:equiv}
Let $F$ be a number field and let $\rhobar: G_F \rightarrow \GL_n(k)$ be continuous and absolutely irreducible. Then	
	\[ \rhounr: G_{F} \rightarrow \GL_n(\Ru) \]
has finite image if and only if the following two properties hold:
	\begin{enumerate}
	\item $\Ru$ is finite over $\OL$;
	\item for any minimal prime $\mathfrak{p}$ of $\Ru[1/p]$, the induced representation $G_F \rightarrow \GL_n(\Ru[1/p]/\mathfrak{p})$ has finite image.
	\end{enumerate}
\end{prop}

\begin{proof}
If $\rhounr$ has finite image, then (2) is clearly satisfied, and (1) follows from Th\'{e}or\`{e}me~2 of \cite{carayol}, which shows that $\Ru$ is generated over $\OL$ by traces. 

Now assume (1) and (2),
and let $E$ be the fraction field of $\OL$. Since
$\Ru$ is a finite $\OL$-algebra,  the map $\Ru \rightarrow \Ru[1/p]$ has finite
kernel. Hence it suffices to prove that the map
$$\rho: G_F \rightarrow \GL_n(\Ru[1/p])$$
has finite image, assuming (2).
Since $\Ru$ is finite over $\OL$, the ring $\Ru[1/p]$ is
a semi-local ring which is a direct sum of Artinian $E$-algebras  $A$ with residue field $H$ for some finite $[H:E] < \infty$.
In particular, the representation $\rho$ breaks up into a finite direct sum of representations to such groups $\GL_n(A)$. 
If $A = H$, then assumption (2)
implies that the image of such a representation is finite.
If $A \ne H$, then $A$ admits a surjective map to $H[\eps]/\eps^2$. In particular, there exists an unramified deformation:
$$\rho: G_{F} \rightarrow \GL_n(H[\eps]/\eps^2).$$
By assumption (2)
again, the corresponding residual representation with image in $\GL_n(H)$ is finite, and is given
by some representation $V$ on which $G_F$ acts through a finite group. Moreover, $\rho$ is then given by some nontrivial extension:
$$0 \rightarrow V \rightarrow W \rightarrow V \rightarrow 0.$$
Consider the restriction of this representation to a finite extension $L/F$ such that $G_{L}$ acts trivially on $V$. Then the action of
$G_L$ on $W$ factors through an unramified $\Z_p$-extension, which must be trivial by class field theory.
It follows that the action of $G_L$ on $W$ is trivial, and hence that the extension $W$ is trivial, a contradiction.
\end{proof}

\subsection{Proof of Corollary~\ref{corr:FM}}

 By Theorem~0.2 of~\cite{Pilloni} (see also~\cite{Kassaei}), one knows  the unramified Fontaine--Mazur conjecture for
 $\rhobar$ under the given hypothesis, hence the result follows from Corollary~\ref{corr:boston}.

\bibliographystyle{amsalpha}
\bibliography{snowden}

 \end{document}